\documentclass{mpi2015-cscpreprint}

\usepackage[american]{babel}
\usepackage{amssymb}
\usepackage{amsthm,amsmath}

\newtheorem{theorem}{Theorem}[section]

\newtheorem{lemma}[theorem]{Lemma}

\newtheorem{remark}{Remark}[section]

\numberwithin{equation}{section}

\begin{document}

\title{Estimates of the Kolmogorov $n$-width for nonlinear transformations with application to distributed-parameter control systems}

\author{Alexander~Zuyev$^{1,2}$, Lihong~Feng$^1$, Peter~Benner}
\affil[1]{Max Planck Institute for Dynamics of Complex Technical Systems, 39106 Magdeburg, Germany}
\affil[2]{Institute of Applied Mathematics and Mechanics, National Academy of Sciences of Ukraine}




\abstract{
This paper aims at characterizing the approximability of bounded sets in the range of
nonlinear operators in Banach spaces by finite-dimensional linear varieties.
In particular, the class of operators we consider includes the endpoint maps of nonlinear distributed-parameter control systems.
{We describe the relationship between the Kolmogorov} $n$-width of a bounded subset and the width of its image {under an} essentially nonlinear transformation.
We propose explicit estimates of the $n$-width in the space of images in terms of the affine part of the corresponding operator and the width of its nonlinear perturbation.
These $n$-width estimates enable us to describe the reachable sets for infinite-dimensional bilinear control systems, with applications to controlling the Euler--Bernoulli beam using a contraction force and to a single-input Schrödinger equation.
}

 \novelty{The key contributions of this paper are as follows:
\begin{itemize}
\item estimate of the Kolmogorov $n$-width of the image of a bounded set under the action of nonlinear operators in a Banach space;
\item decay rate relation for the $n$-width for nonlinear perturbations of an affine operator;
\item inequality for $n$-width of the reachable set for a class of bilinear distributed-parameter systems with bounded controls;
\item reachable set estimates for a controlled Euler--Bernoulli beam model and a bilinear Schrödinger equation.
\end{itemize}}

\maketitle
\section{Introduction}

The Kolmogorov $n$-width, introduced in~\cite{K1936}, plays a remarkable role in approximation theory, serving as a quantifiable measure of the complexity of sets in infinite-dimensional spaces.
In the context of distributed-parameter dynamical systems, which appear in various scientific and engineering disciplines,
the importance of this concept has been pointed out by several authors (see,~e.g.,~\cite{CD2016},~{\cite{DPW2013}},~\cite{UG2019}).
The Kolmogorov $n$-width provides a rigorous way to {implicitly define} the best possible approximation of such systems in a lower-dimensional subspace, offering insights into the inherent approximability of the system. By understanding the behavior of the Kolmogorov $n$-width for a given system, we can identify the dimensions to which the system can be reliably reduced without significant loss of information, thereby facilitating the development of reduced-order models.

Our paper is structured as follows.
The notion of $n$-width with respect to linear varieties and linear subspaces {in a Banach space} is discussed in Section~\ref{sec_auxiliary}, and some useful lemmas for its  characterization are proposed.
Then, the results of Section~\ref{sec_nonlinear} provide estimates {for the image of bounded sets under nonlinear maps consisting of a dominant affine part and a bounded nonlinear part}.
In Section~\ref{sec_bilinear}, the obtained estimates are applied to the endpoint map of a bilinear control system, i.e. to the nonlinear operator acting from the space of control functions to the endpoint of the corresponding trajectory.
This endpoint map is represented by the Volterra series expansion.
This approach allows characterizing the reachable sets of bilinear systems in a given time interval.
As an example, we analyze the $n$-width of the reachable set for the Euler--Bernoulli beam model controlled by a contraction force in Section~\ref{sec_beam}.
In Section~\ref{sec_schrodinger},
we present another example that addresses estimates for the reachable set of a bilinear control system governed by the Schr\"odinger equation.

\section{Notations and auxiliary results}\label{sec_auxiliary}
Let $X$ be a Banach space equipped with the norm $\|\cdot\|_X$.
We denote the distance from a point $x\in X$ to a nonempty set $W\subset X$ as $\rho_X(x,W)=\inf_{\xi\in W} \|x-\xi\|_X$,
and the asymmetric distance between nonempty sets $K\subset X$ and $W\subset X$ is denoted by
$
\rho_X(K,W) = \sup_{x\in K} \rho_X(x,W)
$.
Then the Hausdorff distance is $d_H(K,W)=\max \{\rho_X(K,W),\rho_X(W,K)\}$.
For an $x\in X$ and $\varepsilon>0$, the $\varepsilon$-neighborhood of $x$ is $B_\varepsilon(x)=\{\xi\in X\,\vert\,\| x-\xi \|_X<\varepsilon\}$, and the $\varepsilon$-neighborhood of a set $W\subset X$ is defined as
$
B_\varepsilon(W) = \cup_{x\in W} B_\varepsilon(x)
$.
The set of all linear subspaces of $X$ of dimension at most $n$ is denoted by ${\cal L}_n(X)$,
and the set of all affine subspaces (linear varieties~\cite[Chap.~II, \S 9]{B1989}) of $X$ of dimension at most $n$ is denoted by ${\mathrm{Aff}}_n(X)$.
According to these notations, each affine subspace $W\in {\mathrm{Aff}}_n(X)$ can be obtained by shifting the appropriate linear subspace $W_0\in {\cal L}_n(X)$ by some constant vector $c\in X$: $W = c + W_0$.
For a set {$S\subset X$}, we define its affine span~\cite{R2008} as the set of all (finite) affine combinations of elements of {$S$}, i.e.
{\small
$$
{\mathrm{aff}}(S)=\Bigl\{ \sum_{j=1}^k \alpha_j x_j\,\vert\, k\in{\mathbb N},\, x_j\in S,\, \alpha_j\in{\mathbb R},\,\sum_{j=1}^k\alpha_j=1 \Bigr\}.
$$}
\hspace{-0.5em}The sum of two nonempty sets {$S_1,S_2\subset X$} is understood in the usual sense: {$S_1+S_2=\{x+y\,\vert\,x\in S_1,\,y\in S_2\}$}.

For a set $K\subset X$ and $n\ge 1$, the Kolmogorov $n$-width $D_n(K)_X$ is defined as follows~\cite{K1936,DU2013}:
\begin{equation}
D_n(K)_X := \inf_{W\in {\cal L}_n(X)} \sup_{x\in K} \inf_{\xi\in W} \|x- \xi \|_X.
\label{nwidth_D}
\end{equation}
In addition to the above definition, we also introduce the ``affine'' Kolmogorov $n$-width $d_n(K)_X$ by considering all affine subspaces $W\in {\mathrm{Aff}}_n(X)$~\cite{T1960}:
\begin{equation}
d_n(K)_X := \inf_{W\in {\mathrm{Aff}}_n(X)} \sup_{x\in K} \inf_{\xi\in W} \|x- \xi \|_X.
\label{nwidth}
\end{equation}

According to the above-introduced notations, one can also represent the Kolmogorov widths in the form
{
\begin{equation}
\begin{aligned}
D_n(K)_X &= 
\inf_{W\in {\cal L}_n(X)} \rho_X(K,W),\\
d_n(K)_X &=  
\inf_{W\in {\mathrm{Aff}}_n(X)} \rho_X(K,W).
\end{aligned}
\label{dn_sets}
\end{equation}
It is easy to see that
$
D_n(K)_X \ge d_n(K)_X
\ge D_{n+1}(K)_X\ge d_{n+1}(K)_X$ for each $n\ge 1$},
so that $D_n(K)_X$ and $d_n(K)_X$ are asymptotically equivalent as $n\to\infty$.

We will use the following convenient characterization of $d_n(K)_X$.

\begin{lemma}\label{lem1}
Let $d_n(K)_X $ be the affine Kolmogorov $n$-width of a nonempty set $K\subset X$.
Then:
\begin{enumerate}
\item[1)] for any $\varepsilon>0$, there exists a $W_\varepsilon\in {\mathrm{Aff}}_n(X)$ such that
$K\subset B_r (W_\varepsilon)$ with $r = d_n(K)_X + \varepsilon$;
\item[2)] if there is a number $\bar d_n\ge 0$ such that, for any $r>\bar d_n$, there exists a $W^r\in {\mathrm{ Aff}}_n(X)$ satisfying the property $K\subset B_r (W^r)$,
then $d_n(K)_X\le \bar d_n$.
\end{enumerate}
\end{lemma}

\begin{proof}
Part~1): if $d_n(K)_X$ is defined by~\eqref{nwidth}, then for any $\varepsilon>0$, there is a set $W_\varepsilon\in {\mathrm{Aff}}_n(X)$ such that $\sup_{x\in K} \rho_X(x,W_\varepsilon)<r$, $r=d_n(K)_X+\varepsilon$.
This implies that, for each $x\in K$, there is a $\xi_x\in W_\varepsilon$ such that $x\in B_r(\xi_x)$. Therefore, $K\subset B_r(W_\varepsilon)$.

Part~2): in order to prove that $d_n(K)_X\le \bar d_n$, it remains to show that, for any $\varepsilon>0$, there is a $W_\varepsilon\in{\mathrm{Aff}}_n(X)$ such that $\rho_X(K,W_\varepsilon)<\bar d_n + \varepsilon$ (because of~\eqref{dn_sets}).
The latter inequality holds with $W_\varepsilon=W^r$ and $r=\bar d_n + 2\varepsilon$ under the conditions of part~2). \end{proof}

\begin{remark}
Lemma~\ref{lem1} remains valid under the change of $d_n(K)_X $ and ${\mathrm{Aff}}_n(X)$ to $D_n(K)_X$ and ${\cal L}_n(X)$, respectively.
\end{remark}

{The following lemmas can be easily deduced from the definitions of affine span and the sum of sets; thus, we state them without proof.}

\begin{lemma}\label{lem2}
Let $W^1\in {\mathrm{Aff}}_n(X)$ and $W^2\in {\mathrm{Aff}}_m(X)$. Then
${\mathrm{aff}}(W^1 \cup W^2)\in {\mathrm{Aff}}_{n+m+1}(X)$.
\end{lemma}

\begin{lemma}\label{lem3}
Let $W\in {\cal L}_n(X)$, $c\in X$, $x\in B_{\varepsilon_1}(c+W)$, $y\in B_{\varepsilon_2}(W)$.
Then $x+y \in B_{\varepsilon_1 + \varepsilon_2}(c+W)$.
\end{lemma}

{\em Corollary of Lemma~\ref{lem3}:} {Let $W\in{\cal L}_n(X)$ and $c\in X$.
Then $B_{\varepsilon_1}(c+W)+B_{\varepsilon_2}(W)\subset B_{\varepsilon_1 + \varepsilon_2}(c+W)$ for all $\varepsilon_1,\varepsilon_2>0$.}

\section{Variation of the $n$-width by nonlinear operators}\label{sec_nonlinear}
Let $X$ and $Y$ be Banach spaces, $\pi:X\to Y$ be a nonlinear operator of the form
\begin{equation}
\pi(x) = \ell_0 + \pi_0(x) + F(x),\quad \ell_0 \in Y,
\label{pi_general}
\end{equation}
where $\pi_0: X\to Y$ is a bounded linear operator, and certain growth conditions will be imposed on the nonlinear operator $F:X\to Y$ below.
Throughout the text, we will use the operator norm of $\pi_0$ defined by
$$
\|\pi_0\|=\sup_{x\neq 0} \frac{\|\pi_0(x)\|_Y}{\|x\|_X}.
$$

For a set $K\subset X$, $K\neq \emptyset$, we are interested in comparing $d_n(K)_X$ with the affine Kolmogorov $n$-width of {$\pi(K)$} in the $Y$ space,
i.e. the goal is to estimate
\begin{equation}
\begin{aligned}
d_n(\pi(K))_Y &:= \inf_{W_y\in {\mathrm{Aff}}_n(Y)} \sup_{y\in \pi(K)} \inf_{\eta\in W_y} \|y- \eta \|_Y \\
&= \inf_{W_y\in {\mathrm{Aff}}_n(Y)} \rho_Y(\pi(K),W_y),
\end{aligned}
\end{equation}
where the notations with $Y$ subscript stand for the distance with respect to the norm in $Y$.
In contrast to the paper~\cite{CD2016}, we do not assume that the corresponding maps are holomorphic{.}

\begin{lemma}\label{lem4}
Let $\pi:X\to Y$ be a nonlinear operator of the form~\eqref{pi_general} such that $F:X\to Y$ is bounded on a set $X_0\subset X$.
Then, for any $r_1>0$ and any $W\in {\mathrm{Aff}}_n(X)$, the following property holds:
\begin{equation}
\pi(B_{r_1}(W)\cap X_0) \subset B_{r_2}(W_y),
\label{pi_inclusion}
\end{equation}
where
\begin{equation}
W_y = \ell_0 + \pi_0(W)\in {\mathrm{Aff}}_n(Y),
\label{Wy}
\end{equation}
\begin{equation}
r_2 = r_1  \|\pi_0\| + \sup_{x\in X_0}\| F(x)\|_Y.
\label{r2_ineq}
\end{equation}
\end{lemma}

\begin{proof}
For a given $r_1>0$ and $W\in {\mathrm{Aff}}_n(X)$, our goal is to find an $r_2>0$ such that~\eqref{pi_inclusion} holds, where $W_y=\ell_0 + \pi_0(W)$. Note that $W_y\in {\mathrm{Aff}}_n(Y)$ as the operator $\pi_0:X
\to Y$ is linear. Let $x\in B_{r_1}(W)\cap X_0$. It means that there exist a $\xi\in W$ such that $\|x-\xi\|_X <r_1$ and $x\in X_0$. We estimate the distance between $\pi(x)$ and $\eta = \ell_0 + \pi_0(\xi)\in W_y$:
$$
\begin{aligned}
\|\pi(x)-\eta\|_Y =& \| \pi_0(x) + F(x) - \pi_0(\xi)\|_Y \\
&\le \|\pi_0\| \cdot \|x-\xi\|_X + \|F(x)\|_Y<r_2,
\end{aligned}
$$
provided that the constant $r_2$ is defined in~\eqref{r2_ineq}. We have shown that $\pi(x)\in B_{r_2}(W_y)$, which proves~\eqref{pi_inclusion}.
\end{proof}

\begin{remark}
If $W\in {\cal L}_n(X)$, then the affine subspace $W_y$ defined by~\eqref{Wy} may not belong to ${\cal L}_n(Y)$. This motivates our estimation of the affine $n$-width $d_n(\pi(K))_Y$ instead of $D_n(\pi(K))_Y$ in the following theorem.
\end{remark}

\begin{theorem}\label{thm1}
Let $K\neq \emptyset$ be a bounded subset of $X$, and
let $\pi:X\to Y$ be a nonlinear operator of the form~\eqref{pi_general} such that $F:X\to Y$ is bounded on $K$.
Then
\begin{equation}
d_n (\pi(K))_Y \le \|\pi_0\| d_n (K)_X  + \sup_{x\in K}\|F(x)\|_Y.
\label{dn_est}
\end{equation}
\end{theorem}

\begin{proof}
Since $K$ is bounded, $d_n(K)_X<\infty$ and $d_n(\pi(K))_Y<\infty$ under the assumptions of this theorem.

Due to Lemma~\ref{lem1}, for each $\varepsilon>0$, there exists a $W_\varepsilon\in {\mathrm{Aff}}_n(X)$ such that
$K\subset B_{r_1} (W_\varepsilon)$ with $r_1 = d_n(K)_X + \varepsilon$.
Moreover, Lemma~\ref{lem4} implies that
\begin{equation}
\pi(K) = \pi(B_{r_1}({W_\varepsilon})\cap K) \subset B_{r_2}(W_y),
\label{r1r2_inclusion}
\end{equation}
where $W_y\in {\mathrm{Aff}}_n(Y)$ and $r_2=    r_1 \|\pi_0\|+ \sup_{x\in K}\| F(x)\|_Y  $.

Then, by exploiting inclusion~\eqref{r1r2_inclusion} and part~2) of Lemma~\ref{lem1} with $\bar d_n = \|\pi_0\| d_n (K)_X  + \sup_{x\in K}\|F(x)\|_Y$,
we obtain estimate~\eqref{dn_est}.
\end{proof}

\begin{remark}
Assume that $0\in K$, $F(0)=0$, and that the operator $F$ satisfies the Lipschitz condition on $K$, i.e. there exists a constant $\mu\ge 0$:
$$
\|F(x')-F(x'')\|_Y \le \mu \|x'-x''\|_X\quad\text{for all}\;\; x',x''\in K.
$$
Then it is clear that inequality~\eqref{dn_est} yields the estimate
$$
d_n (\pi(K))_Y \le \|\pi_0\| d_n (K)_X  + \mu \, \sup_{x\in K}\|x\|_X.
$$
\end{remark}

\begin{theorem}\label{thm2}
Let $\pi:X\to Y$ be a nonlinear operator of the form~\eqref{pi_general}, and let $K\neq \emptyset$ be a subset of $X$. Then the following estimate holds for each $n,m\in \mathbb N$:
\begin{equation}
d_{n+m}(\pi(K))_Y \le \|\pi_0\| d_n(K)_X + D_m (F(K))_Y.
\label{d_D_estimate}
\end{equation}
\end{theorem}

\begin{proof}
For an arbitrary $\varepsilon>0$, we consider the sets $W_\varepsilon^\pi\in{\mathrm{Aff}}_n (X)$ and $W_\varepsilon^F\in{\cal L}_m (Y)$ such that $K\subset B_{r_x}( W_\varepsilon^\pi )$ and $F(K) \subset B_{r_y} (W_\varepsilon^F) $, where
$r_x = d_n(K)_X+\varepsilon$ and $r_y= D_m(F(K))_Y + \varepsilon$. Such sets $W_\varepsilon^\pi$ and $W_\varepsilon^F$ exist because of Lemma~\ref{lem1}.

Then we define $W_y = \ell_0 + \pi_0(W_\varepsilon^\pi)\in{\mathrm{Aff}}_n(X)$ and note that $W_y = c+ W_y^0$ for some $c\in Y$ and $W_y^0 \in {\cal L}_n(Y)$.
Due to the choice of $W_\varepsilon^\pi$, for any $x\in K$, there is a $\xi\in W_\varepsilon^\pi$ such that $\|x-\xi\|_X<r_x$.
By applying the affine operator $\pi_1(x) = \ell_0 + \pi_0(x)$ to $x$ and $\xi$, we get:
$$
\|\pi_1(x)-\pi_1(\xi)\|_Y \le \|\pi_0\|\cdot \|x-\xi\|_X < r_x\|\pi_0\| .
$$
As $\eta = \pi_1(\xi)\in W_y$, the above inequality implies that $\pi_1(x) \in B_{r_x\|\pi_0\| }(W_y )$ for each $x\in K$,
i.e. $\pi_1(K) \subset B_{r_x \|\pi_0\| }(W_y )$.

Let $\tilde W = W_y^0 + W_\varepsilon^F\in {\cal L}_{n+m}(Y)$. Because $  W_\varepsilon^F\subset\tilde W$ and  $W_y^0 \subset \tilde W$, we have:
$F(K)\subset B_{r_y}(\tilde W)$ and $\pi_1(K)\subset  B_{r_x \|\pi_0\| } (c+\tilde W)$.
Then the corollary of Lemma~\ref{lem3} implies that
\begin{equation}
\pi(K) \subset \pi_1(K) + F(K) \subset B_r(c+\tilde W)\in {\mathrm{Aff}}_{n+m}(Y),
\label{inclusion_W}
\end{equation}
where
{\small
$$
r= r_x \|\pi_0\| + r_y = \|\pi_0\| d_n(K)_X + D_m (F(K))_Y + \varepsilon(1+\|\pi_0\|).
$$}
\hspace{-0.4em}As $\varepsilon>0$ can be chosen arbitrarily, the obtained inclusion~\eqref{inclusion_W} together with part~2) of Lemma~\ref{lem1} yields the assertion of Theorem~\ref{thm2}.
\end{proof}

\begin{remark}
The results of this section are valid for Banach spaces over both the field of real numbers and the field of complex numbers. In Section~\ref{sec_schrodinger}, we will highlight the application of this theory in complex-valued spaces.
\end{remark}

\section{Bilinear control systems}\label{sec_bilinear}
Consider a single-input control system in a Hilbert space~$H$:
\begin{equation}
\begin{aligned}
&\dot x(t) = (A+u(t)B)x(t),\quad x(t)\in H,\;t\in [0,T],\\
&x(0)=x_0\in H,
\label{bilinear_cs}
\end{aligned}
\end{equation}
where $A:D(A)\to H$ is the infinitesimal generator of a $C_0$-semigroup $\{e^{tA}\}_{t\ge 0}$ of bounded linear operators on $H$,
$B:H\to H$ is a bounded linear operator, and $u(\cdot)\in L^2(0,T)$ is the control.

The solutions of the Cauchy problem~\eqref{bilinear_cs} are represented by the Volterra series (cf.~\cite{lamnabhi1995volterra,GGB2015}) as follows:
\begin{equation}
\begin{aligned}
x(t) = & e^{tA}x_0 +  \sum_{k=1}^\infty V_k(t;u),\quad t\in [0,T],\\
V_1(t;u)=&  \int_0^t w_1(t,\sigma_1) u(\sigma_1) d\sigma_1,\\
V_k(t;u)= & \int_0^t \int_0^{\sigma_k} ... \int_0^{\sigma_2} w_k(t,\sigma_1,...,\sigma_k) \\
& \times u(\sigma_1) ... u(\sigma_k) d\sigma_1 ... d\sigma_k,\; k=2,3,...,
\end{aligned}
\label{Volterra}
\end{equation}
where the kernels are
\begin{equation}
w_k(t,\sigma_1,...,\sigma_k) = e^{(t-\sigma_k)A} B e^{(\sigma_k-\sigma_{k-1})A} B ... B e^{\sigma_1 A} x_0,
\label{kernels}
\end{equation}
for $0\le \sigma_1 \le ...\le \sigma_k\le t$.
Note that the above expansion can also be adapted for distributed parameter systems with more general nonlinearities. In particular, a local representation of the solutions by Volterra series for the Schr\"odinger partial differential equation with cubic nonlinearity has been developed in~\cite{GGBL2015}.
A procedure for computing the Volterra kernels for a nonlinear Burgers' equation with the fractional derivative of order $1/2$ is described in~\cite{HH2004}.

For a given set of admissible controls $K\subset L^2(0,T)$ and the initial value $x_0\in H$, the reachable set of system~\eqref{bilinear_cs} at time $T$ is defined as
$$
{\cal R}_T(x_0) = \{x(T)\,\vert \, x(0)=x_0,\,u(\cdot) \in K\}\subset H.
$$
Note that system~\eqref{bilinear_cs} is never controllable in $H$ because its equilibrium $x=0$ is an invariant set,
with
${\cal R}_T(0)=\{0\}$ for each $T>0$.
{
The reachable sets of infinite-dimensional systems of the form~\eqref{bilinear_cs} may exhibit highly complex structures (see, e.g.,~\cite{BR2014} and references therein).
This complexity remains prominent in systems that are neither exactly nor approximately controllable.
In this context, understanding the sensitivity of the reachable set
under a specific
class of admissible controls becomes crucial.
We explore this sensitivity using the concept of Kolmogorov $n$-width, which formalizes the relationship between the $n$-width of the admissible controls and that of the corresponding reachable set.
A direct estimation of $d_n({\cal R}_T(x_0))_H$ itself (without its relation to $d_n(K)_{L^2(0,T)}$) fails to reveal detailed insights into this sensitivity and the underlying complexity.
This motivates our present analysis, where
}
we apply the results of Section~\ref{sec_nonlinear} to characterize the $n$-width of ${\cal R}_T(x_0)$.

\begin{theorem}\label{thm3}
Let $\|e^{tA}\|\le M e^{\omega t}$ with some constants $M\ge 1$, $\omega\in\mathbb R$, and let $K\subset L^2(0,T)$ be a nonempty set of admissible controls such that $\|u(\cdot)\|_{L^2(0,T)}\le r$ for each $u(
\cdot)\in K$. If $M \|B\| \sqrt{T}r<1$, then
\begin{equation}
\begin{aligned}
d_n({\cal R}_T(x_0))_H \le & M^2 \|B\| \|x_0\| \sqrt{T} e^{\omega T} d_n(K)_{L^2(0,T)} \\
& + \frac{M^3 \|B\|^2 \|x_0\| T e^{\omega T}r^2}{1-M \|B\| \sqrt{T}r}\;\text{for each}\;n\in{\mathbb N}.
\end{aligned}
\label{dn_bilinear}
\end{equation}
\end{theorem}

\begin{proof}
Let us define the operator $\pi:L^2(0,T)\to H$ by the rule
$
\pi(u) = \ell_0 + \pi_0(u) + F(u)
$,
where
$$
\begin{aligned}
\ell_0 &= e^{TA}x_0, \\
\pi_0(u) &= \int_0^T  e^{(T-\sigma_1)A}B e^{\sigma_1 A} x_0 u(\sigma_1) d\sigma_1,\\
F(u) &= \sum_{k=2}^\infty\int_0^T \int_0^{\sigma_k} ... \int_0^{\sigma_2} w_k(T,\sigma_1,...,\sigma_k) \\
& \times u(\sigma_1) \cdots u(\sigma_k) d\sigma_1 ... d\sigma_k,
\end{aligned}
$$
and the kernels $w_k$ are defined in~\eqref{kernels}. By construction,
$
{\cal R}_T(x_0) = \{ \pi(u)\,\vert\, u\in K\}
$. By estimating the norm of $\pi_0(u)$ with the use of the Cauchy--Schwarz inequality, we conclude that
$$
\begin{aligned}
&\|\pi_0(u)\|  \le \int_0^T \|e^{(T-\sigma_1)A}Be^{\sigma_1 A}x_0\|\cdot | u(\sigma_1)| d\sigma_1 \\
&\le \left(\int_0^T \|e^{(T-\sigma_1)A}Be^{\sigma_1 A}\|^2 d\sigma_1\right)^{1/2}\|x_0\| \cdot \|u(\cdot)\|_{L^2(0,T)}.
\end{aligned}
$$
This implies
\begin{equation}
\|\pi_0\| \le M^2 \|B\| \|x_0\| \sqrt{T} e^{\omega T}.
\label{pi0ineq}
\end{equation}
{T}he kernels $w_k$ defined in equation~\eqref{kernels} are estimated as
\begin{equation}\label{kernel_est}
\| w_k(t,\sigma_1,...,\sigma_k) \| \le \|B\|^k M^{k+1} \|x_0\| e^{\omega t},\; \quad k=2,3,\,...\, ,
\end{equation}
and, for $t\in [0,T]$, the iterated integrals satisfy the property
\begin{equation}\label{int_est}
\int_0^t\int_0^{\sigma_k} ... \int_0^{\sigma_2} |u(\sigma_1) \cdots u(\sigma_k)| d\sigma_1 ... d\sigma_k \le t^{\frac{k}{2}} \|u(\cdot)\|^k_{L^2(0,t)}.
\end{equation}
The above estimate follows from the Cauchy--Schwarz inequality:
{\small
$$
\int_0^t |u(\sigma)|d\sigma \le \left(\int_0^t d\sigma\right)^{\frac12} \left(\int_0^t u^2(\sigma)d\sigma\right)^{\frac12}
 = \sqrt{t}\|u(\cdot)\|_{L^2(0,t)}.
$$}

By combining inequalities~\eqref{kernel_est} and~\eqref{int_est}, we conclude that
$$
\|V_k(t;u)\| \le M e^{\omega t} \|x_0\| \left(M\|B\|\sqrt{t}\|u(\cdot)\|_{L^2(0,T)}\right)^k,
$$
for all $k=2,3,...$ .
Thus, if $M \|B\| \sqrt{T}r<1$, then
\begin{equation}
\|F(u)\| \le \frac{M^3 \|B\|^2 T e^{\omega T} r^2}{1-M \|B\| \sqrt{T}r}\|x_0\|.
\label{Fineq}
\end{equation}
Finally, Theorem~\ref{thm1} together with inequalities~\eqref{pi0ineq},~\eqref{Fineq} implies the assertion of Theorem~\ref{thm3}.
\end{proof}
{
\begin{remark}\label{rem_T}
The upper bounds of the iterated integrals~\eqref{int_est} increase polynomially as $t$ increases.
Consequently, the  estimate obtained in~\eqref{dn_bilinear}  depends essentially on the time horizon $T$, under the assumptions of Theorem~\ref{thm3}.
\end{remark}
\begin{remark}\label{rem_const}
The first term of the right-hand side of~\eqref{dn_bilinear} tends to zero as $d_n(K)_{L^2(0,T)}\to 0$ with $n\to \infty$.
For large $n$, the  estimate~\eqref{dn_bilinear} is dominated by its last term, which is proportional to the initial value norm $\|x_0\|$.
Hence,  Theorem~\ref{thm3} can be used if the influence of the last term in~\eqref{dn_bilinear} is acceptably small in the context of a required description of the reachable set.
The asymptotic behavior of $d_n({\cal R}_T(x_0))_H$ can be further refined by Theorem~\ref{thm2}, where the right-hand side of~\eqref{d_D_estimate} tends to zero for large $n$ and $m$.
We leave the  potential extensions of Theorem~\ref{thm2}, which involve characterizing endpoint maps for nonlinear control systems, to future research.
\end{remark}
}

\section{Application to the Euler--Bernoulli beam controlled by a contraction force}\label{sec_beam}

\subsection{PDE-based model}
Consider a controlled Euler--Bernoulli beam model of length $L$ with simply supported boundary conditions:
\begin{equation}\label{EB_PDE}
\frac{\partial^2 w}{\partial t^2}+ a^2 \frac{\partial^4 w}{\partial z^4} = u(t) \frac{\partial}{\partial z} \left( \chi_{[z_1,z_2]}(z)\frac{\partial w}{\partial z}\right),\; z\in(0,L),
\end{equation}
\begin{equation}\label{EB_BC}
\left.w\right|_{z=0} = \left.w\right|_{z=L} = 0,\; \left. \frac{\partial^2 w}{\partial z^2}\right|_{z=0} =
\left. \frac{\partial^2 w}{\partial z^2}\right|_{z=L} = 0,
\end{equation}
where $w=w(t,z)$ represents the transverse displacement of the beam's centerline at a position $z\in[0,L]$ and time~$t$,
$a=\sqrt{EI/\rho_0}>0$, $E$ is the Young's modulus, $I$ denotes the moment of inertia of the beam's cross-section,
$\rho_0$ is the mass per unit length of the beam. The control input $u(t)$ is proportional to the contraction force $f(t)$ applied to a segment of the beam defined by $z\in [z_1,z_2]$: $u(t)=f(t)/\rho_0$, and $\chi_{[z_1,z_2]}(z)$ is the indicator function of the interval $[z_1,z_2]$ within the beam's length $[0,L]$.
The derivative of the indicator function in the right-hand side of~\eqref{EB_PDE} is not formally defined at $z=z_1$ and $z=z_2$.
Therefore, in further analysis, $\chi_{[z_1,z_2]}(z)$ can either be replaced with a differentiable approximation or the problem defined in equations~\eqref{EB_PDE}--\eqref{EB_BC} can be addressed using a weak formulation or Galerkin's method.

Note that the partial differential equation~\eqref{EB_PDE} and the boundary conditions~\eqref{EB_BC} can be derived from Hamilton's principle, which involves the kinetic energy $\cal K$ and the potential energy $\cal U$:
$$
{\cal K} = \frac12 \int_0^L \rho_0 \left(\frac{\partial w}{\partial t}\right)^2 dz,\;
{\cal U} =\frac12 \int_0^L EI \left(\frac{\partial^2 w}{\partial z^2}\right)^2 dz.
$$
Indeed, if $w(t,z)$ describes the motion of the considered beam under the action of the force $f(t)$ for $t\in [0,T]$, then
the variation of the action functional, according to Hamilton's principle~\cite{B2009}, satisfies:
\begin{equation}\label{Hamilton}
\delta \int_{0}^{T} ({\cal K}-{\cal U})dt + \int_{0}^{T} \delta {\cal A}\,dt =0,
\end{equation}
for each admissible variation $\delta w = \delta w(t,z)$ that vanishes at $t=0$, $t=T$, $z=0$, and $z=L$.
The work $\delta {\cal A}$ done by the control force $f(t)$ is
\begin{equation}\label{dA}
\delta {\cal A} = - f(t)\int_0^L \chi_{[z_1,z_2]}(z) \frac{\partial w(t,z)}{\partial z} \frac{\partial\, \delta w(t,z)}{\partial z} dz.
\end{equation}
This action can be implemented by a piezoactuator laminated to the beam between $[z_1,z_2]$, assuming the thickness is negligible.
Values of $f(t)>0$ correspond to a contraction force, while $f(t)<0$ represents an expansive force.
In cases where the beam's thickness is not negligible, the bending moment induced by the piezoactuator can be modeled using a linear Euler--Bernoulli beam equation (see, for example,~\cite{HM2001}).
The bilinear control system~\eqref{EB_PDE}--\eqref{EB_BC} is derived from the variational formulation~\eqref{Hamilton}--\eqref{dA}  by utilizing the integration by parts technique and applying the fundamental lemma of calculus of variations.

\subsection{Galerkin's method}
For further application of Galerkin's method and the results developed in the previous section, we
introduce the linear differential operator $U:D(U)\to L^2(0,L)$ such that
\begin{equation}\label{U_op}
\begin{aligned}
 (U \varphi) (z) = \frac{d^4 \varphi(z)}{dz^4},\, & \; z\in(0,L),\\
 D(U)=\bigl\{ \varphi\in H^4(0,L)\,\vert\,  & \varphi(0)=\varphi(L)=0,\\
 & \varphi''(0)=\varphi''(L)=0 \bigr\}.
\end{aligned}
\end{equation}
The eigenvalues and normalized eigenfunctions of $U$ are, respectively,
\begin{equation}\label{eig}
\lambda_n = \left(\frac{\pi n}{L}\right)^4,\; \varphi_n(z) = \sqrt{\frac{2}L}\sin\left(\frac{\pi n z}{L}\right),\;
n=1,2,...\; .
\end{equation}
It is well-known~\cite{LGM1999} that $\{\varphi_n\}_{n=1}^\infty$ is an orthonormal basis in $L^2(0,L)$.
We seek a solution $w(t,z)$ to the problem defined by equations~\eqref{EB_PDE}--\eqref{EB_BC} in the form of a series expansion
$$
w(t,z) = \sum_{j=1}^\infty q_j(t) \varphi_j(z).
$$
We apply Galerkin's method (cf.~\cite{Zu2015}) by substituting the series $w(t,z)$ into equation~\eqref{EB_PDE} and multiplying both sides of equation~\eqref{EB_PDE} by $\varphi_n(z)$.
After computing the inner products in the $L^2(0,L)$ space and leveraging the orthogonality of eigenfunctions~\eqref{eig}, we obtain the following infinite system of differential equations:
\begin{equation}\label{q_sys}
\ddot q_n(t) = -\omega_n^2 q_n(t) + u(t) \sum_{j=1}^\infty {\omega_j} b_{nj} q_j(t),\;n=1,2,...\, ,
\end{equation}
where
\begin{equation}\label{omega_b}
\begin{aligned}
\omega_n = a \left(\frac{\pi n}{L}\right)^2>0,\,
 b_{nj} = -\frac{1}{\omega_j}\left<\chi_{[z_1,z_2]} \varphi_j',\varphi_n'\right>_{L^2(0,L)}.
 \end{aligned}
\end{equation}
To rewrite the control system~\eqref{q_sys} in an abstract form, we introduce
the infinite vectors
$
\xi(t)=(\xi_1(t),\xi_2(t),...)^\top\in \ell^2$ and $\eta(t)=(\eta_1(t),\eta_2(t),...)^\top\in \ell^2$,
with the condition that their components satisfy:
$
\xi_n(t) = \omega_n q_n(t)$, $\eta_n(t) = \dot q_n(t)$, $ n=1,2,$ ... .
As a result, the bilinear control system~\eqref{q_sys} takes the following form:
\begin{equation}\label{beam_abstract}
\dot x(t) = (A+ u(t) B)x(t), \quad x(t)\in H=\ell^2 \times \ell^2,
\end{equation}
where the state vector is $x(t)=\begin{pmatrix}\xi(t) \\ \eta(t)\end{pmatrix}\in H$ and the operators $A$ and $B$ are given by their infinite block matrices
\begin{equation}\label{x_A_B}
A = \begin{pmatrix} 0 & \Omega \\ - \Omega & 0\end{pmatrix},\; B = \begin{pmatrix}0 & 0 \\
b & 0\end{pmatrix},
\end{equation}
$$
\Omega = \begin{pmatrix}\omega_1 & 0 & ... \\
0 & \omega_2 & ...\\
\vdots & \vdots & \ddots
\end{pmatrix},\; b= \begin{pmatrix} b_{11} & b_{12} & \dots \\
b_{21} & b_{22} & \dots \\
\vdots & \vdots & \ddots
\end{pmatrix}.
$$
We treat the actions $Ax(t)$ and $Bx(t)$ in the sense of usual matrix-vector multiplication, where the summation is taken over the set of positive integer indices.
The components $b_{nj}$ in~\eqref{omega_b} are computed using formulas~\eqref{eig}:
$$
b_{nj} = - \frac{n}{a\pi j} \int_0^\pi \chi_{[z_1,z_2]}\left(\frac{sL}\pi\right)\bigl(\cos(n+j)s + \cos(n-j)s\bigr) ds.
$$
By taking into account that $\chi_{[z_1,z_2]}(z)$ is the indicator function, we rewrite the above expression as
\begin{equation}\label{b_nj}
\begin{aligned}
& b_{nj} = - \frac{n}{a\pi j} \int_{\pi z_1/L}^{\pi z_2/L} \bigl(\cos(n+j)s + \cos(n-j)s\bigr)ds\\
& = \frac{n}{a\pi j (j-n)}
\left\{ \sin\left( \frac{(j-n)\pi z_1}L \right) + \sin\left(\frac{(n - j)\pi z_2}L\right) \right\} \\
&+  \frac{n}{a\pi j (j+n)}  \left\{ \sin\left(\frac{(n + j)\pi z_1}L \right) - \sin\left(\frac{(n + j) \pi z_2}L\right) \right\},
\end{aligned}
\end{equation}
for all $j\neq n$, and the diagonal components of $b$ are
\begin{equation}\label{b_nn}
b_{nn} = \frac{z_1-z_2}{a L} + \frac{1}{2 a \pi n} \left\{ \sin\left( \frac{2 \pi n z_1}L \right) - \sin\left( \frac{2 \pi n z_2}L \right) \right\}.
\end{equation}

\subsection{Application of Theorem~\ref{thm3}}

In order to apply Theorem~\ref{thm3}, we first check that the operator $A:D(A)\to H$, presented in~\eqref{x_A_B}, generates a $C_0$-semigroup of bounded linear operators on the real Hilbert space $H=\ell^2\times \ell^2$, where the inner product in $H$ is defined in the usual way:
$
\left<x,\tilde x,\right>_H = \left<\xi,\tilde \xi\right>_{\ell^2} + \left<\eta,\tilde \eta\right>_{\ell^2} = \sum_{n=1}^\infty (\xi_n {\tilde \xi}_n + \eta_n {\tilde \eta}_n)
$.
Indeed, the domain of definition of the operator $A$ in~\eqref{x_A_B} is
$$
D(A) = \left\{ \begin{pmatrix}\xi \\ \eta \end{pmatrix}\in H\,|\, \sum_{n=1}^\infty \omega_n^2(\xi_n^2 + \eta_n^2)<\infty\right\},
$$
and $A$ generates the $C_0$-semigroup of bounded linear operators $\{e^{tA}\}_{t\ge 0}$ on $H$, according to the Lumer--Phillips theorem~\cite{LGM1999}.
This semigroup is extended to the unitary group  $\{e^{tA}\}_{t\in {\mathbb R}}$, with the operators
 $e^{tA}$ being represented by the following block-diagonal matrix:
$$
\begin{aligned}
e^{tA} = {\textrm diag}\left(e^{tA_1},e^{tA_2},...\right),\,
e^{tA_n}=\begin{pmatrix}\cos \omega_n t & \sin \omega_n t \\ -\sin\omega_n t & \cos \omega_n t\end{pmatrix},
\end{aligned}
$$
\begin{equation}\label{etA_norm}
\|e^{tA}\| = 1 \quad \text{for all}\; t\in\mathbb R.
\end{equation}

In the sequel, we restrict our consideration to the case $[z_1,z_2]=[0,L]$, so that $\chi_{[z_1,z_2]}(z)\equiv 1$ on $[0,L]$, and the operator $B:H\to H$ in~\eqref{x_A_B} is reduced to the following form:
$
B =-\frac{1}{a} \begin{pmatrix}0 & 0 \\ I & 0\end{pmatrix}
$,
owing to the representations~\eqref{b_nj} and~\eqref{b_nn}.
Here, $I$ is the identity operator on $\ell^2$ and, therefore,
\begin{equation}\label{B_norm}
\|B\| = \frac{1}a.
\end{equation}

By putting together the norm representations~\eqref{etA_norm} and~\eqref{B_norm}, we deduce the following corollary of Theorem~\ref{thm3} concerning the reachable sets of system~\eqref{beam_abstract}.

{\em Corollary~1 of Theorem~\ref{thm3}:}
Let $[z_1,z_2]=[0,L]$, and let $K$ be a nonempty subset of $L^2(0,T)$ such that $\|u(\cdot)\|_{L^2(0,T)}\le r$ for all $u(\cdot)\in K$. Then, for any initial state $x_0\in H$ and $T<(a/r)^2$, the Kolmogorov $n$-width of the corresponding reachable set ${\cal R}_T(x_0)$ of system~\eqref{beam_abstract} with controls from $K$ admits the estimate:
\begin{equation}
\begin{aligned}
d_n({\cal R}_T(x_0))_H \le & \frac{\|x_0\| \sqrt{T} }{a} d_n(K)_{L^2(0,T)} \\
&+ \frac{\|x_0\| T r^2}{a(a- \sqrt{T}r)}\;\text{for each}\;n\in{\mathbb N}.
\end{aligned}
\label{dn_beam}
\end{equation}

\section{Application to the bilinear Schr\"odinger equation}\label{sec_schrodinger}
The obtained results will be applied to estimate the reachable set of the controlled
Schr\"odinger equation describing a quantum particle in a one-dimensional potential well~\cite{R2003,BR2014}:
\begin{equation}
\begin{aligned}
& i \frac{\partial \psi(t,z)}{\partial t} = - \frac{\partial^2 \psi(t,z)}{\partial z^2} - u(t)\mu(z)\psi(t,z),\quad z\in (0,1),\\
& \psi(t,0)=\psi(t,1)=0,\quad t\in (0,T),
\end{aligned}
\label{schrodinger}
\end{equation}
where $\psi:[0,T]\times [0,1]\to \mathbb C$ is the wave function, $\mu$ is the dipolar moment of the particle, $u$ is the amplitude of the electric field{.}
In this section, we denote the Lebesgue and Sobolev spaces over the field of complex numbers as $L^2 \left( (a,b);{\mathbb C} \right)$ and $H^k \left( (a,b);{\mathbb C} \right) $, respectively.
The brief notations $L^2(a,b)$ and $H^k(a,b)$ are reserved for the spaces over the field of real numbers.

It was proved in~\cite{BL2010} that system~\eqref{schrodinger} is locally controllable in $$H^3_{(0)}(0,1)=\{\psi \in H^3 \left( (0,1);{\mathbb C} \right) \,|\, \psi = \psi'' =0 \;\text{at} \;z=0\;\text{and}\; z=1\},$$
but it is uncontrollable in $H^2\left( (0,1);{\mathbb C} \right)\cap H^1_0\left( (0,1);{\mathbb C} \right)$ in the exact sense.
More precisely, system~\eqref{schrodinger} with controls in $L^2(0,T)$ is locally exactly controllable around its ground state in $H^3_{(0)}(0,1) \cap {\mathcal S}$ under generic assumptions on $\mu\in H^3(0,T)$, where $S$ is the unit sphere in $L^2 \left( (0,1);{\mathbb C} \right)$.
In addition, some states of $H^3_{(0)}(0,1) \cap {\mathcal S}$ are note reachable in arbitrary small time $T$ under appropriate conditions on $\mu$, so a positive minimal time $T$ is required for local controllability~\cite{C2006}.

In this section, we estimate the Kolmogorov $n$-width of the reachable sets ${\cal R}_T$ for system~\eqref{schrodinger} in the Hilbert space $H=L^2\left((0,1);{\mathbb C}\right)$
by considering bounded controls of the class $L^2(0,T)$.
Control system~\eqref{schrodinger} can be written in the abstract form~\eqref{bilinear_cs} as
\begin{equation}
\dot \phi (t) = A\phi(t) +u(t) B \phi(t),
\label{schrodinger_abstract}
\end{equation}
where $A:D(A)\to H$ and $B:H\to H$ are defined by
$$
\begin{aligned}
D(A) &= H^2\left((0,1);{\mathbb C}\right) \cap H^1_0\left((0,1);{\mathbb C}\right),\\
A \phi &= i \frac{d^2 \phi}{dx^2},\; B \phi = i \mu \phi,
\end{aligned}
$$
and $\|B\|=\|\mu\|_{L^2(0,1)}$.
It is well-known~\cite{BL2010} that $A$ generates the group $\{e^{tA}\}_{t\in\mathbb R}$ of isometries on $H$, and the action of $e^{tA}$ on a $\phi \in H$ is defined as
$
e^{tA} \phi = \sum_{k=1}^\infty \left< \phi,\phi_k\right>e^{-i\lambda_k t } \phi_k
$,
where $\phi_k(z)=\sqrt{2} \sin(\pi k z)$ are normalized eigenfunctions of $A$ and
$\lambda_k = \pi^2 k^2$, $k\in\mathbb N$.
By applying Theorem~\ref{thm3} to control system~\eqref{schrodinger_abstract}, we obtain the following corollary.

{\em Corollary~2 of Theorem~\ref{thm3}:}
Let $K$ be a nonempty subset of $L^2(0,T)$ such that $\|u(\cdot)\|_{L^2(0,T)}\le r$ for all $u(\cdot)\in K$. Then, for any initial state $\phi_0\in H$ and $T<\frac{1}{\|\mu\|^2 r^2}$, the Kolmogorov $n$-width of the corresponding reachable set ${\cal R}_T(\phi_0)$ of system~\eqref{schrodinger_abstract} with controls from $K$ admits the following estimate:
\begin{equation}
\begin{aligned}
d_n({\cal R}_T(\phi_0))_H \le & \|\mu\| \|\phi_0\| \sqrt{T} d_n(K)_{L^2(0,T)} \\
&+ \frac{\|\mu\|^2 \|\phi_0\| T r^2}{1-\|\mu\| \sqrt{T}r}\;\text{for each}\;n\in{\mathbb N}.
\end{aligned}
\label{dn_schrodinger}
\end{equation}

\section{Conclusions and future work}
To the best of our knowledge, this paper is the first work dealing with estimates of the reachable sets of infinite-dimensional control systems in terms of the Kolmogorov $n$-width.
The results of Section~\ref{sec_nonlinear} are applicable to rather general classes of nonlinear operators {between Banach spaces $X$ and $Y$}.
{
Our characterization differs from the approaches of~\cite{DPW2013} for compact sets and~\cite{CD2016} for holomorphic maps,
as we do not assume the differentiability of the corresponding operator $\pi:X\to Y$ or the compactness of a set $K\subset X$.
Although our results do not yield a straightforward polynomial or exponential decay of the $n$-width as $n\to \infty$, we propose explicit estimates of $d_n(\pi(K))_Y$ in terms of the norm (Theorem~\ref{thm1}) and the width (Theorem~\ref{thm2}) of nonlinear terms.
}
The endpoint maps, considered in Section~\ref{sec_bilinear}, represent only one possible application of our approach.
Another prospective development could tailor the $n$-width theory to analyzing residuals of finite difference methods used in solving nonlinear partial differential and differential-algebraic equations.

It should be emphasized that the estimates~\eqref{dn_beam} and~\eqref{dn_schrodinger}, obtained as Corollaries~1 and~2 or Theorem~\ref{thm3},
are valid for arbitrary bounded sets of admissible controls $K$.
More specific assumptions, such as the compactness of {$K$} or the decay rate of $d_n(K)_{L^2(0,T)}$, are not required in Theorem~\ref{thm3}.
 We do not explore possible refinements of inequalities~\eqref{dn_bilinear},~\eqref{dn_beam},~\eqref{dn_schrodinger} under additional assumptions on $K$, deferring a more subtle analysis of the decay rate of $d_n({\cal R}_T(x_0))_H$, {also with potential extensions of Theorem~\ref{thm2},} to future studies.


\end{document}